\documentclass[11pt]{amsart}

\usepackage[a4paper,hmargin=3.5cm,vmargin=4cm]{geometry}
\usepackage{amsfonts,amssymb,amscd,amstext,verbatim}

\usepackage{fancyhdr}
\usepackage[utf8]{inputenc}
\usepackage{hyperref}
\usepackage{verbatim}
\input xy
\xyoption{all}

\usepackage{times}

\usepackage{enumerate}
\usepackage{titlesec}
\usepackage{mathrsfs}

\titleformat{\section}
{\filcenter\bfseries\large} {\thesection{.}}{0.2cm}{}
\titleformat{\subsection}[runin]
{\bfseries} {\thesubsection{.}}{0.15cm}{}[.]
\titleformat{\subsubsection}[runin]
{\em}{\thesubsubsection{.}}{0.15cm}{}[.]

\usepackage[up,bf]{caption}

\newtheorem{theorem}{Theorem}[section]
\newtheorem{proposition}[theorem]{Proposition}

\newtheorem{lemma}[theorem]{Lemma}
\newtheorem{corollary}[theorem]{Corollary}

\theoremstyle{definition}
\newtheorem{definition}[theorem]{Definition}
\newtheorem{remark}[theorem]{Remark}

\newtheorem{problem}[theorem]{Problem}
\newtheorem{example}[theorem]{Example}

\numberwithin{equation}{section}
\numberwithin{figure}{section}

\newcommand\Ascr{\mathscr{A}}
\newcommand\Bscr{\mathscr{B}}
\newcommand\Cscr{\mathscr{C}}

\newcommand\Hscr{\mathscr{H}}

\newcommand\Oscr{\mathscr{O}}

\newcommand\C{\mathbb{C}}

\newcommand\CP{\mathbb{CP}}

\newcommand\N{\mathbb{N}}

\newcommand\R{\mathbb{R}}

\newcommand\igot{\mathfrak{i}}

\renewcommand\igot{\mathfrak{i}}

\renewcommand\imath{\igot}

\newcommand\hra{\hookrightarrow}

\newcommand\wt{\widetilde}
\newcommand\wh{\widehat}
\newcommand\di{\partial}
\newcommand\dibar{\overline\partial}

\newcommand\dist{\mathrm{dist}}

\newcommand\Id{\mathrm{Id}}

\def\dist{\mathrm{dist}}

\newcommand\Hi{H^\infty}
\newcommand\HiM{H^\infty_M}

\numberwithin{equation}{section}

\begin{document}

\fancyhead[LO]{Euclidean domains in complex manifolds}
\fancyhead[RE]{F.\ Forstneri{\v c}} 
\fancyhead[RO,LE]{\thepage}

\thispagestyle{empty}

\vspace*{1cm}
\begin{center}
{\bf\LARGE Euclidean domains in complex manifolds}

\vspace*{0.5cm}

{\large\bf  Franc Forstneri{\v c}} 
\end{center}

\vspace*{1cm}

\begin{quote}
{\small
\noindent {\bf Abstract}\hspace*{0.1cm}
In this paper we find big Euclidean domains in complex manifolds.
We consider open neighbourhoods of sets $K\cup M$ in a complex manifold $X$, 
where $K$ is a compact holomorphically convex set in an open Stein neighbourhood, 
$M$ is an embedded Stein submanifold, and $K\cap M$ is compact $\Oscr(M)$-convex. 
We prove a Docquier--Grauert type theorem concerning biholomorphic equivalence of
neighbourhoods of such sets, and we give sufficient conditions for the existence of 
Stein neighbourhoods of $K\cup M$, biholomorphic to domains in $\C^n$ with $n=\dim X$, 
such that $M$ is mapped onto a closed complex submanifold of $\C^n$.
\vspace*{0.2cm}

\noindent{\bf Keywords}\hspace*{0.1cm} Stein manifold, Oka manifold, polynomially convex set

\vspace*{0.1cm}

\noindent{\bf MSC (2020):}\hspace*{0.1cm}} 
32E10; 32E20; 32E30; 32H02; 32Q56 

\vspace*{0.1cm}
\noindent{\bf Date: \rm September 13, 2021}

%
%
%
%
\end{quote}

%
%
%
%
\section{Introduction}\label{sec:intro}
Complex Euclidean spaces $\C^n$  are the most basic objects in analytic and algebraic geometry.
It is of major interest to understand which properties of a complex manifold make it a Euclidean space, 
or an open subset thereof. There are outstanding conjectures regarding the former question,    
{\em the recognition problem in complex analysis}. In this paper we consider the second question, 
which is of both theoretical interest and practical importance, especially since it is substantially easier 
to solve analytic problems on Euclidean spaces than on general complex manifolds. 
In particular, we prove the following result.

%
%
\begin{theorem}\label{th:main}
Assume that $X$ is a complex manifold, $M$ is a Stein submanifold of $X$ such that 
$TX|_M$ is a trivial bundle, $K$ is a compact set in $X$ such that $K\cap M$ is compact and
$\Oscr(M)$-convex, $\Omega_0\subset X$ is an open neighbourhood of $K$, and 
$\Phi_0:\Omega_0\stackrel{\cong}{\to} \Phi_0(\Omega_0) \subset\C^n$ with $n=\dim X$ 
is a biholomorphic map such that $\Phi_0(K)$ is polynomially convex in $\C^n$. 
If $\dim X \ge 2\dim M+1$ then for any $\epsilon>0$ there are a Stein neighbourhood $\Omega \subset X$ 
of $K\cup M$ and a biholomorphic map $\Phi:\Omega \stackrel{\cong}{\to} \Phi(\Omega) \subset \C^n$ 
such that $\sup_{x\in K}|\Phi(x)-\Phi_0(x)| <\epsilon$ and $\Phi(M)$ is a closed complex submanifold of $\C^n$.
If $\dim X=2\dim M$ then $\Phi$ can be chosen an immersion which is proper on $M$
and satisfies $\Phi(\Omega\setminus K)\subset \C^n\setminus \Phi(K)$.
\end{theorem}

As usual, $\Oscr(X)$ denotes the algebra of holomorphic functions on a complex manifold $X$. 
The {\em holomorphic hull} of a compact set $K$ in $X$ is the set
\[ 
	\wh K_{\Oscr(X)} =\{x\in X: |f(x)|\le \max_K |f|\ \ \text{for all}\ f\in \Oscr(X)\}.
\] 
We say that $K$ is {\em $\Oscr(X)$-convex} (or holomorphically convex in $X$) 
if $K=\wh K_{\Oscr(X)}$. A holomorphically convex set in $\C^n$ is called {\em polynomially convex}. 
An analogous result holds with $\C^n$ replaced by a 
Stein manifold having the density property; see Theorem \ref{th:density}.

The theorem says in particular that $K\cup M$ admits a Euclidean Stein 
neighbourhood in $X$ which makes $M$ a proper submanifold of $\C^n$. 
It also gives a partial result on the open problem whether every Stein manifold $X$ 
of dimension $n>1$ with trivial tangent bundle admits a holomorphic immersion $X\to\C^n$; 
see \cite[p.\ 70]{Gromov1986} and \cite[Problem 9.13.3]{Forstneric2017E}. 

Theorem \ref{th:main} easily reduces to the case when $X$ is a Stein manifold, $M$ is closed in $X$, 
and $K$ is a compact $\Oscr(X)$-convex strongly pseudoconvex domain; see Lemma \ref{lem:admissible}. 
For such sets, holomorphic convexity is a stable property under $\Cscr^2$ deformations; 
see \cite[Corollary 1.4]{Forstneric1986PAMS}. Assuming that the boundary $bK$ intersects $M$ transversely, 
the same holds for $\Oscr(M)$-convexity of $K\cap M$. 
This makes the conditions in Theorem \ref{th:main} quite robust. 
In this case, there is a holomorphic map $\Phi:X\to\C^n$ which satisfies the conclusion of 
Theorem \ref{th:main} on a  neighbourhood of $K\cup M$ and also 
$\Phi(X\setminus K)  \subset \C^n\setminus \Phi(K)$ (see Theorem \ref{th:schlicht}).

Triviality of the restricted tangent bundle $TX|_M$ is clearly a necessary condition in Theorem \ref{th:main}. 
In the special case when $M$ is a complex curve (the image of an open Riemann surface), 
every holomorphic vector bundle over $M$ is trivial by the Oka--Grauert principle
(see Grauert \cite{Grauert1958MA} or \cite[Theorem 5.3.1]{Forstneric2017E}). 
However, for manifolds $M$ of dimension $\ge 2$ this is a nontrivial condition irrespectively 
of the codimension of $M$ in $X$. For example, if the tangent bundle $TM$ is nontrivial, 
then the direct sum of $TM$ and a trivial bundle $M\times\C^k$ of arbitrary rank is also nontrivial since 
we are in the stable range (see \cite[Theorem 8.3.8]{Forstneric2017E}), and hence $M\times \{0\}^k$ 
has no Euclidean neighbourhood in $X=M\times \C^k$. 

The proof of Theorem \ref{th:main} (see Section \ref{sec:proof}) involves three ingredients. 
The Oka--Grauert theory is used to deal with questions concerning the normal bundle 
of a Stein submanifold in a complex manifold (see Section \ref{sec:normalbundle}).
The second ingredient of independent interest
generalizes the Docquier--Grauert tubular neighbourhood theorem for Stein manifolds; 
see Theorem \ref{th:tubular}.  The third ingredient is a method from the papers 
\cite{DrinovecForstneric2010AJM,ForstnericRitter2014},  
which enables us to make the submanifold $\Phi(M)\subset \C^n$ proper in $\C^n$
and such that $\Phi(M\setminus K)\subset \C^n\setminus \Phi(K)$.

We shall consider configurations of the following kind. 

%
%
\begin{definition}\label{def:admissible}
A pair $(K,M)$ of subsets of a complex manifold $X$ is an {\em admissible pair} if 
\begin{enumerate}[\rm (a)]
\item $K$ is compact and has a Stein neighbourhood $U\subset X$ such that  $K$ is $\Oscr(U)$-convex
(such $K$ is said to be {\em holomorphically convex}),
\item $M$ is a (not necessarily closed) embedded Stein submanifold of $X$, and 
\item $K\cap M$ is a compact $\Oscr(M)$-convex subset of $M$.
\end{enumerate}
\end{definition}

Note that a pair $(K,M)$ in Theorem \ref{th:main} is admissible.
For later reference, we recall the following result from \cite[Theorem 1.2]{Forstneric2005AIF} 
(see also \cite[Theorem 3.2.1]{Forstneric2017E}); the case $K=\varnothing$
corresponds to Siu's theorem \cite{Siu1976}.

%
%
\begin{theorem}\label{th:Steinneighbourhoods}
If $(K,M)$ is an admissible pair in a complex manifold $X$, then $K\cup M$ has a basis of 
open Stein neighbourhoods $V \subset X$ such that $M$ is closed in $V$ 
and $K$ is $\Oscr(V)$-convex.
\end{theorem}
 
Given a complex submanifold $M$ of a complex manifold $X$, we denote by 
\[ 
	 \nu_{M,X} = TX|_M / TM
\] 
the holomorphic normal bundle of $M$ in $X$. 
More generally, a holomorphic immersion $f:M\to X$ induces a holomorphic vector bundle embedding
$\iota:TM\hra f^*TX$ by $\iota_x(v)=f^*(df_x(v))$ for $x\in M$ and $v\in T_x M$,
and the normal bundle of the immersion is 
\[
	\nu_f = f^*TX/\iota(TM).
\]
(Here, $f^*$ denotes the pullback).
Assume now that $M$ is a Stein submanifold of $X$. Then, the normal bundle
$\nu_{M,X}$ embeds as a holomorphic vector subbundle of $TX|_M$ such that 
$TX|_M=TM\oplus\nu_{M,X}$ (see \cite[Corollary 2.6.6]{Forstneric2017E}).
Furthermore, a theorem of Docquier and Grauert \cite{DocquierGrauert1960} 
(see also \cite[Theorem 3.3.3]{Forstneric2017E}) says that the natural inclusion of $M$ onto the zero section 
of the normal bundle $\nu_{M,X}$ extends to a biholomorphic map from an open neighbourhood of 
$M$ in $X$ onto an open neighbourhood of the zero section in $\nu_{M,X}$. 
(The extra assumption in \cite{DocquierGrauert1960} that the manifold $X$ be Stein
is unnecessary in view of Theorem \ref{th:Steinneighbourhoods}. 
See \cite[Chapter 3]{Forstneric2017E} for more on this subject.) 
Similarly, an immersion $f:M\to X$ from a Stein manifold $M$ extends
to a holomorphic immersion $F:\Omega\to X$ from a neighbourhood $\Omega\subset \nu_f$ of the
zero section of $\nu_f$ (which we identify with $M$) into $Y$.

The following result generalizes the Docquier--Grauert tubular neighbourhood theorem
to admissible pairs. This provides an important ingredient in the proof of Theorem \ref{th:main}. 

%
%
\begin{theorem}
\label{th:tubular}
Let $(K,M)$ be an admissible pair in a complex manifold $X$ (see Definition \ref{def:admissible}). 
Assume that $Y$ is a complex manifold with $\dim Y=\dim X$, $\Omega_0\subset X$ is an open 
neighbourhood of $K$, and $\Phi_0 : \Omega_0 \cup M\to Y$ is a such that 
$\Phi_0|_{\Omega_0}:\Omega_0 \to \Phi_0(\Omega_0) \subset Y$ is
a biholomorphism and $f:=\Phi_0|_M:M\to Y$ is a holomorphic immersion.
Assume that there is a topological isomorphism $\Theta: \nu_{M,X}\to \nu_{f}$ 
of the normal bundles which is given by the differential of $\Phi_0$ over $K\cap M$. 
Given $\epsilon>0$ there are an open Stein neighbourhood 
$\Omega \subset X$ of $K\cup M$ and a holomorphic immersion 
$\Phi:\Omega \to Y$ such that 
\[
	\Phi|_M=f=\Phi_0|_M
	\quad \text{and}\quad
	\sup_{x\in K} \dist_Y\left(\Phi(x),\Phi_0(x)\right) <\epsilon.
\]
If in addition the map $\Phi_0 : \Omega_0 \cup M\to Y$ is injective and $f:=\Phi_0|_M:M\to N:=f(M)\subset Y$ 
is a holomorphic embedding, then $\Phi:\Omega\to \Phi(\Omega)\subset Y$ can be chosen biholomorphic. 
\end{theorem}

\vspace{-4mm}
\[
	\xymatrix{   \nu_{M,X} \ar[d] \ar[r]^{\Theta}  & \nu_{N,Y} \ar[d]  \\
				   M          \ar[r]^{f}      & N
	}			   
\]
\vspace{2mm}

The classical Docquier--Grauert theorem 
(cf.\ \cite{DocquierGrauert1960} or \cite[Theorem 3.3.3]{Forstneric2017E}) 
corresponds to the special case of Theorem \ref{th:tubular} with $K=\varnothing$.
Theorem \ref{th:tubular} is proved in Section \ref{sec:proof-tubular}. An important ingredient 
in the proof is a new version of the splitting lemma for biholomorphic maps 
close to the identity on a Cartan pair (cf.\ \cite[Theorem 9.7.1]{Forstneric2017E}) 
with interpolation on a complex submanifold; see Theorem \ref{Acta:theorem4.1}
and Remark \ref{rem:variety}. 

Combining Theorem \ref{th:tubular} with the proper holomorphic embedding theorem 
for Stein manifolds, due to Eliashberg and Gromov \cite{EliashbergGromov1992AM} 
and Sch{\"u}rmann \cite{Schurmann1997} (see also \cite[Theorem 9.3.1]{Forstneric2017E}), 
gives the following analogue of Theorem \ref{th:main} in which $K=\varnothing$ 
and $M$ may have smaller codimension in $X$. See  Section \ref{sec:normalbundle} for the details. 

%
%
\begin{theorem} \label{th:Kempty}
Assume that $M$ is a Stein submanifold of a complex manifold $X$ 
such that the restricted tangent bundle $TX|_M$ is trivial.
\begin{enumerate}[\rm (a)]  
\item If $m:=\dim M\ge 2$ and $n:=\dim X \ge \left[\frac{3m}{2}\right]+1\ge 4$,  
then there are a Stein neighbourhood $\Omega\subset X$ of $M$ and a biholomorphic map 
$\Phi:\Omega \stackrel{\cong}{\to} \Phi(\Omega) \subset \C^n$ 
such that $\Phi(M)$ is a closed complex submanifold of $\C^n$.
\vspace{1mm}
\item If $m\ge 1$ and $n \ge \left[\frac{3m +1}{2}\right]\ge 2$, then there are a Stein neighbourhood 
$\Omega\subset X$ of $M$ and a holomorphic immersion $\Phi:\Omega \to \C^n$ 
such that $\Phi|_M:M\to\C^n$ is proper.
\end{enumerate}
\end{theorem}

By an example of Forster \cite{Forster1970} (see Example \ref{ex:Forster}),
Theorem \ref{th:Kempty} is dimensionwise sharp for a general manifold $M$ of the given dimension, 
except perhaps when $M$ is an open Riemann surface embedded in a complex surface $X$. 
In the latter case, the answer to part (a) of the theorem hinges upon the classical open problem 
(the Forster--Bell--Narasimhan Conjecture) 
whether every open Riemann surface embeds properly holomorphically into $\C^2$.
See \cite[Sections 9.10--9.11]{Forstneric2017E} for a the survey of this topic. 

We do not see a comparably simple proof of Theorem \ref{th:main} 
in the presence of a nonempty compact set $K\subset X$. The following remains an open problem.

%
%
\begin{problem}\label{prob:Knotempty}
Assuming that the set $K$ in Theorem \ref{th:main} is nonempty, does the conclusion hold under the weaker 
assumption $\dim X \ge \left[\frac{3\dim M}{2}\right]+1\ge 4$ in Theorem \ref{th:Kempty}?
\end{problem}

%
%
\section{Euclidean neighbourhoods of Stein submanifolds}\label{sec:normalbundle}

In this section we recall some results of Oka--Grauert theory which will be used, 
and we prove Theorem \ref{th:Kempty}. 
Among the references for Oka--Grauert theory, see Grauert \cite{Grauert1958MA} and the surveys 
by Leiterer \cite{Leiterer1986} and the author \cite[Chapters 5 and 8]{Forstneric2017E}.
All complex spaces in this paper are assumed to be reduced.

Recall that any complex vector bundle over a Stein space $X$ has a compatible structure of a 
holomorphic vector bundle, and any topological complex vector bundle isomorphism between 
a pair of holomorphic vector bundles over $X$ is isotopic to a holomorphic vector bundle isomorphism.
In fact, essentially any problem concerning holomorphic vector bundles over Stein spaces
reduces to the corresponding topological problem for the underlying complex vector bundles.

Let  $\pi : E\to X$ and $\pi': E'\to X$ be complex vector bundles of rank $r$ and $r'$, 
respectively. A complex vector bundle map $\Phi:E\to E'$ is said to be 
{\em of maximal rank} if for each $x\in X$ the $\C$-linear map $\Phi_x:E_x\to E'_x$
is of maximal rank equal to $\min\{r,r'\}$. If $r=r'$ then such a map is a complex vector bundle isomorphism, 
for $r<r'$ it is a complex vector bundle embedding, and for $r>r'$ it is a complex vector bundle
epimorphism. We shall need the following result; see \cite[Theorem 8.3.1]{Forstneric2017E}.

%
%
%
%
\begin{theorem} \label{th:maxrank}
Let $\pi : E\to X$ and $\pi': E'\to X$ be holomorphic vector bundles of rank $r$ and $r'$, 
respectively, over a reduced Stein space $X$. 
If $K$ is a compact $\Oscr(X)$-convex subset of $X$ and  $X'$ is 
a closed complex subvariety of $X$, then the following hold.
\begin{enumerate}[\rm (a)]
\item A topological complex vector bundle map $\Theta: E\to E'$ of maximal rank is homotopic
through maps of the same type to a holomorphic vector bundle map. 
If $\Theta$ is holomorphic over $X'$ and on a neighbourhood of $K$, 
then the homotopy can be chosen fixed on $X'$, holomorphic on a neighbourhood of $K$, 
and uniformly close to $\Theta$ on $K$.
\vspace{1mm}
\item 
If $|r-r'|\ge \left[ \frac{\dim X}{2}\right]$ then there exists a holomorphic vector bundle map $\Theta:E\to E'$ of
maximal rank. Furthermore, given a holomorphic vector bundle map $\Theta_0: E\to E'$ of maximal rank
over a neighbourhood of $K$ and over the subvariety $X'$, we can choose $\Theta$ as above such that it agrees
with $\Theta_0$ over $X'$ and it approximates $\Theta_0$ over $K$.
\end{enumerate}
\end{theorem}

Item (a) in Theorem \ref{th:maxrank} corresponds to items (a) and (b) in 
\cite[Theorem 8.3.1]{Forstneric2017E}, and the interpolation in item (b) is 
part (c) of the cited theorem. The approximation statement in (b) follows 
from the proof given in \cite[p.\ 361]{Forstneric2017E}, which depends on  
\cite[Corollary 5.14.3]{Forstneric2017E}.

Part 1 of the following corollary is an immediate application of Theorem \ref{th:maxrank} (b).
Part 2 follows by taking $X=M\times \C$, $X'=M\times \{0,1\}$,  
$E=\pi^*TM$ where $\pi:X=M\times \C \to M$ is the projection onto the first factor, 
and $E'=\pi^*(F) \to X$. 

%
%
\begin{corollary}\label{cor:isotopicembeddings}
Let $M$ be a Stein manifold and $F\to M$ be a holomorphic vector bundle.
\begin{enumerate}[\rm 1.]
\item If $\mathrm{rank}\, F \ge \left[\frac{3\dim M}{2}\right]$ then there exists a holomorphic vector bundle 
embedding $\iota:TM\hra F$. Furthermore, if $K$ is a compact $\Oscr(M)$-convex subset of $M$, $U\subset M$ 
is an open neighbourhood of $K$, and $M'$ is a closed complex subvariety of $M$, 
then for any holomorphic vector bundle embedding $\iota_0:TM|_{U\cup M'} \hra F|_{U\cup M'}$ 
there is a global holomorphic vector bundle embedding $\iota:TM\hra F$ which agrees with 
$\iota_0$ over $M'$ and approximates $\iota_0$ as closely as desired over $K$.
\vspace{1mm}
\item If $\mathrm{rank}\, F \ge \left[\frac{3\dim M+1}{2}\right]$ then any two holomorphic vector bundle 
embeddings $\iota_0,\iota_1: TM\hra F$ are isotopic through a family of holomorphic vector bundle 
embeddings $\iota_t: TM\hra F$ with $t\in \C$. Hence, the quotient bundles $F/\iota_t(TM)$ for $t\in \C$ 
are pairwise isomorphic holomorphic vector bundles. 
\end{enumerate}
\end{corollary}

The following immediate consequence of Corollary \ref{cor:isotopicembeddings} 
says that, under the stated dimension conditions, the isomorphism class of the normal bundle 
$\nu_f$ of a Stein immersion $f:M\to X$ is uniquely determined by the isomorphism class of 
the bundle $f^*TX$.

%
%
\begin{corollary} \label{cor:stablenormabundle}
Suppose that $M$ is a Stein manifold of dimension $m$.
If $f_i:M\to X_i$ for $i=0,1$ are holomorphic immersions into complex manifolds 
of the same dimension $n \ge \left[\frac{3m+1}{2}\right]$ such that $f^*_0(TX_0)\cong f^*_1(TX_1)$ 
are isomorphic complex vector bundles over $M$, then the normal bundles
$\nu_{f_i}$ for $i=0,1$ are also isomorphic.
\end{corollary}

Corollary \ref{cor:stablenormabundle}, together with the tubular neighbourhood theorem of 
Docquier and Grauert \cite{DocquierGrauert1960} (see also \cite[Theorem 3.3.3]{Forstneric2017E}), 
obviously imply the following result.

%
%
\begin{proposition}\label{prop:stablenormalbundle}
Let $M$ be a Stein manifold of dimension $m$, and let $X_0, X_1$ be complex manifolds of
dimension $n\ge \left[\frac{3m+1}{2}\right]$. If $f_0:M\hra X_0$ is a holomorphic embedding
and $f_1:M\to X_1$ is a holomorphic immersion such that $f^*_0(TX_0)\cong f^*_1(TX_1)$ 
are isomorphic complex vector bundles over $M$, then there are an open Stein neighbourhood
$\Omega_0\subset X_0$ of $f_0(M)$ and a holomorphic immersion 
$\Phi:\Omega_0\to X_1$ such that $f_1=\Phi\circ f_0$ holds on $M$.  
If $f_1$ is an embedding then $\Phi$ can be chosen a biholomorphism onto its image.
\end{proposition}

%
%
\begin{proof}[Proof of Theorem \ref{th:Kempty}]
By a seminal theorem of Eliashberg and Gromov \cite{EliashbergGromov1992AM} and
Sch{\"u}rmann \cite{Schurmann1997} (see also \cite[Theorem 9.3.1]{Forstneric2017E}), a
Stein manifold $M$ of dimension $m=\dim M\ge 2$ admits a proper holomorphic embedding $M\hra \C^n$ 
for any $n\ge \left[\frac{3m}{2}\right]+1$, and a proper holomorphic immersion 
$M\to  \C^n$ for any $n\ge \left[\frac{3m+1}{2}\right]$, where the latter also holds for $m=1$.
Since we assumed that the restricted tangent bundle $TX|_M$ is trivial, 
the conclusions of the theorem follow from Proposition \ref{prop:stablenormalbundle}. 
\end{proof}

%
%
The following example shows that triviality of the restricted tangent bundle $TX|_M$ along a Stein submanifold
$M\subset X$ does not suffice for the existence of a Euclidean neighbourhood of $M$ in $X$ if the codimension
of $M$ is too small in comparison to $\dim M$.

%
%
\begin{example}\label{ex:Forster}
Let $M$ be the Stein surface 
\[ 
    	M = \bigl\{ [x : y : z]\in \CP^2 :  x^2+y^2+z^2\ne 0 \bigr\}.
\]
It was shown by Forster \cite[p.\ 183]{Forster1970} that 
the tangent bundle of $M$ is nontrivial. 
Thus, $M$ does not admit a proper holomorphic embedding into $\C^{3}$ (since any closed complex 
hypersurface in $\C^n$ has trivial tangent bundle), and it does not admit a holomorphic immersion into $\C^2$.
However, like any Stein surface, it admits a proper holomorphic immersion $M\to\C^3$.
Let $X\to M$ be the normal line bundle of this immersion, so $TM\oplus X$ is isomorphic to the trivial bundle 
$M\times \C^3$. Note that $X$ is nontrivial since we are in the stable
range, i.e., triviality of $X$ would imply triviality of $TM$ (see \cite[Corollary 8.3.9]{Forstneric2017E}). 
This means that $X$ is a Stein threefold containing a properly embedded copy of $M$, 
namely its zero section, such that $TX|_M\cong TM\oplus X$ is trivial, but $M$ has no 
Euclidean neighbourhood in $X$ which would make $M$ proper in $\C^3$.    
However, we do not know whether $M$ admits a {\em nonproper} holomorphic embedding into $\C^3$. 
(The argument of Forster relies on the fact that {\em closed} complex hypersurfaces in $\C^n$ are parallelizable,
so properness of $M$ in $\C^3$ is essential.) Clearly,  such an embedding would give a Euclidean 
neighbourhood of $M$ inside $X$.

Similar examples exist in all dimensions $m=\dim M\ge 2$. Indeed, with $M$ as above, it suffices to take 
$S=M^{m/2}$ (the Cartesian product of $m/2$ copies of $M$) if $m$ is even,
and $S=M^{(m-1)/2}\times \C$ if $m$ is odd; see \cite[p.\ 183]{Forster1970}.
Another family of examples of Stein manifolds, which do not immerse or embed properly holomorphically into
$\C^n$ below the threshold dimension in the Eliashberg--Gromov--Sch\"urmann theorem,  
can be found in the paper \cite{HoJacobowitzLandweber2012} by Ho, Jacobowitz, and Landweber. 
\qed\end{example}

%
%
\section{A splitting lemma with interpolation for biholomorphic maps}\label{sec:preliminaries}

In this section we prepare some technical results which will be used in the proof of
Theorem \ref{th:tubular}. The main result of the section is Theorem \ref{Acta:theorem4.1}. 

We recall the following; see \cite[Lemma 6.5]{Forstneric1999JGEA} for the basic
case when $X=\C^n$ and note that the same proof applies in general. 

%
%
\begin{lemma}\label{lem:hc}
Let $X$ be a Stein manifold, $K$ be a compact $\Oscr(X)$-convex set in $X$, 
and $X'$ be a closed complex subvariety of $X$. For any compact $\Oscr(X')$-convex subset 
$L \subset X'$ such that $K\cap X' \subset L$, the union $K\cup L$ is $\Oscr(X)$-convex. 
\end{lemma}

Next, we show that the set $K$ in Theorem \ref{th:main} can be replaced 
by a somewhat bigger compact strongly pseudoconvex domain satisfying the same hypotheses.

%
%
\begin{lemma}\label{lem:admissible}
If $(K,M)$ is an admissible pair in a complex manifold $X$ (see Definition \ref{def:admissible}),
then for every open set $U\subset X$ containing $K$ there exists a strongly pseudoconvex domain 
$D\subset X$ such that $K\subset D\subset \overline D\subset U$ and $\overline D\cap M$ is 
$\Oscr(M)$-convex. If in addition $\Phi:U\stackrel{\cong}{\to} \Phi(U)\subset \C^n$ $(n=\dim X)$ 
is a biholomorphic map such that $\Phi(K)$ is polynomially convex, then $D$ can be chosen such 
that $\Phi(\overline D)$ is also polynomially convex.
\end{lemma}

\begin{proof}
Let us first prove the full assertion of the lemma, including the one in the last sentence.
Assume that $U$ and $\Phi$ are as in the last sentence, so $\Phi(K)$ is polynomially convex in $\C^n$.
By \cite[Proposition 2.5.1]{Forstneric2017E} 
there is a smooth plurisubharmonic exhaustion function 
$\rho:\C^n\to\R_+$ such that $\rho^{-1}(0)=\Phi(K)$ and $\rho$ is strongly plurisubharmonic on 
$\C^n\setminus \Phi(K)$. Hence, there is a constant $c_0>0$ such that for every $c\in (0,c_0]$ the set 
\[
	\Omega_c=\{x\in U: \rho\circ\Phi(x)<c\} \Subset U
\] 
is a strongly pseudoconvex neighbourhood of $K$ such that 
$\Phi(\overline \Omega_c)=\{\rho\le c\}$ is polynomially convex in $\C^n$,  
and $\Omega_c$ has smooth boundary for almost all such $c$.
The function $u=\rho\circ \Phi|_{M\cap U}:M\cap U\to \R_+$ is plurisubharmonic and vanishes 
precisely on $K\cap M$. Since $K\cap M$ is $\Oscr(M)$-convex, 
there is a plurisubharmonic exhaustion function $v:M\to\R$ such that 
$v<0$ on $K\cap M$ and $v>0$ on $M\setminus \Omega_{c_0}$ (see \cite[Theorem 5.1.6]{Hormander1990}).
We replace $v$ by $h\circ v$, where $h:\R\to\R_+$ is a smooth convex increasing function with 
$h(t)=0$ for $t\le 0$, $h(t)>0$ for $t>0$, and $\lim_{t\to\infty}h(t)=+\infty$. 
The new function $v$ is a nonnegative plurisubharmonic exhaustion function on $M$ that vanishes on a neighbourhood of $K\cap M$ and is positive on $b\Omega_{c_0}$. 
Let $c_1=\min\{v(x): x\in b\Omega_{c_0}\}>0$. If $c_2>0$ is chosen such that 
$c_1c_2>\max\{u(x): x\in b\Omega_{c_0}\}$, then the function $\phi:M\to\R_+$ given by 
\[
	\phi(x)=\begin{cases} \max\{u(x),c_2 v(x)\}, & \text{if $x\in \Omega_{c_0}$}; \\
					 c_2v(x), & \text{if $x\in M\setminus  \Omega_{c_0}$}
		  \end{cases}			 
\]
is a plurisubharmonic exhaustion function on $M$ that agrees with $u$ on the set $\{v=0\}$,  
which is a neighbourhood of $K\cap M$. Hence for all sufficiently small $c>0$ we have that
\[
	\overline\Omega_c\cap M= \{x\in M\cap U: u(x)\le c\} = \{x\in M:\phi\le c\},
\]
and this set is $\Oscr(M)$-convex by \cite[Theorem 5.2.10]{Hormander1990}. 
For every such $c$ for which $b\Omega_c$ is smooth (i.e., $c$ is a regular value of $\rho$), 
the domain $D=\Omega_c$ satisfies the conclusion.

In general, 
we pick a smooth plurisubharmonic 
function $\rho\ge 0$ on a Stein open neighbourhood $U\subset X$ of $K$ such that 
$\rho^{-1}(0)=K$ and $\rho$ is strongly plurisubharmonic on $U\setminus K$ 
(see \cite[Proposition 2.5.1]{Forstneric2017E}). The above proof then shows that for any 
$c>0$ which is a regular value of $\rho$, the set $D=\{\rho<c\}$ satisfies the desired conclusion. 
\end{proof}

In the proof of Theorem \ref{th:tubular} we shall use the following 
notion from \cite[Definition 5.7.1]{Forstneric2017E}.

%
%
\begin{definition} \label{def:Cartan-pair}
A pair $(A,B)$ of compact subsets in a complex manifold $X$
is a {\em Cartan pair} if it satisfies the following two conditions:    
\begin{enumerate} [\rm(i)]
\item  The sets $D=A\cup B$ and $C=A\cap B$ are Stein compacts (i.e., they have  
bases of open Stein neighbourhoods), and
\item $A$ and $B$ are {\em separated} in the sense that
$\overline{A\setminus B}\cap \overline{B\setminus A} =\varnothing$.
\end{enumerate}
\end{definition}

By \cite[Proposition 5.7.3]{Forstneric2017E}, one can approximate a Cartan pair
from the outside by Cartan pairs in which $A,B,C,D$ 
are closures of strongly pseudoconvex Stein domains in $X$. 

%
%

Let $\pi:E\to X$ be a holomorphic hermitian vector bundle on a complex manifold $X$, and let
$M$ be a closed complex subvariety of $X$. Given an open set $U\subset X$ we consider  
the Banach spaces 
\begin{equation}\label{eq:notation}
	\Gamma^\infty_M(U,E) \subset \Gamma^\infty(U,E), 
\end{equation}
where $\Gamma^\infty(U,E)$ consists of all bounded holomorphic sections $U\to E|_U$ 
and is endowed with the supremum norm (using the given hermitian metric on $E$), 
and $\Gamma^\infty_M(U,E)$ denotes its closed subspace consisting of sections vanishing on $M\cap U$. 
Both spaces are nontrivial if $X$ is Stein and $U$ is relatively compact in $X$.

We shall need the following linear splitting lemma with interpolation on a subvariety.

%
%
\begin{lemma}\label{lem:linearsplitting}
Assume that $X$ is a Stein manifold, $\pi:E\to X$ is a holomorphic hermitian vector bundle, 
$M$ is a closed complex subvariety of $X$, and $A, B$ are open
relatively compact sets in $X$ satisfying the following conditions:
\begin{enumerate}[\rm (i)]
\item $D=A\cup B$ is a smoothly bounded strongly pseudoconvex domain in $X$,
\item $\overline{A\setminus B} \cap \overline{B\setminus A} =\varnothing$
(the separation condition), and
\item $M$ has no singularities on $bD$ and intersects $bD$ transversely.
\end{enumerate}
Then there exist bounded linear operators 
\[
	\Ascr : \Gamma^\infty_M(C,E) \to \Gamma^\infty_M(A,E),\qquad
	\Bscr : \Gamma^\infty_M(C,E) \to \Gamma^\infty_M(B,E)
\]
such that $c=\Bscr(c) - \Ascr(c)$ holds on $C$ for each $c\in \Gamma^\infty_M(C,E)$.
\end{lemma}

\begin{proof}
Denote by $H^\infty(U)$ the Banach space of bounded holomorphic functions on $U\subset X$ 
endowed with the sup-norm, and by $H^\infty_M(U)$ its closed subspace consisting of function 
vanishing on $M\cap U$. For the trivial bundle $E=X\times \C$ (i.e., for functions) the lemma follows from  
\cite[proof of Lemma 3.2]{ForstnericPrezelj2001}, except that we replace the use of 
\cite[Lemma 3.1]{ForstnericPrezelj2001} by Henkin's 
bounded linear extension operator $S: H^\infty(D\cap M) \to H^\infty(D)$ satisfying
\begin{equation}\label{eq:S}
	(Sh)(x)=h(x)\ \ \text{for all $h\in H^\infty(D\cap M)$ and $x\in D\cap M$}.
\end{equation}
(See Henkin \cite{Henkin1972} or Henkin and Leiterer 
\cite[Theorem 4.11.1 and Remark on p.\ 196]{HenkinLeiterer1984}.
In the cited sources, the subvariety $M$ is assumed to be without 
singularities, and we shall only use Lemma \ref{lem:linearsplitting} in this particular case.
However, our assumptions imply that $M\cap D$ has at most finitely many singularities,
and an inspection of \cite[proof of Theorem 4.11.1]{HenkinLeiterer1984}
shows that it caries over to this situation. Indeed, the main action takes place near
the boundary points of $D$, where $M$ is nonsingular.) We first find bounded linear operators 
\begin{equation}\label{eq:A0B0}
	\Ascr_0: H^\infty(C) \to H^\infty(A),\qquad  \Bscr_0: H^\infty(C) \to H^\infty(B)
\end{equation}
such that 
\[ 
	c=\Bscr_0(c) - \Ascr_0(c)\ \ \text{holds on $C$ for each $c\in H^\infty(C)$}.
\] 
This is a special case of the Cousin--I problem with bounds, which is solved by using 
the separation condition (ii) and the existence of a sup-norm bounded linear solution 
operator to the $\dibar$-equation on $D$ at the level of $(0,1)$-forms; 
see \cite[Lemma 2.4]{ForstnericPrezelj2000}. Assume now that $c=0$ on $C\cap M$.
Then, $\Bscr_0(c)-\Ascr_0(c)=c$ vanishes on $C\cap M$, so the pair 
$(\Ascr_0(c)|_{A\cap M},\Bscr_0(c)|_{B\cap M})$ amalgamates into 
a bounded holomorphic function $\Hscr(c) \in H^\infty(D\cap M)$. This defines
a bounded linear operator $\Hscr:H^\infty_M(C)\to H^\infty(D\cap M)$. 
Let $S: H^\infty(D\cap M) \to H^\infty(D)$ be Henkin's bounded extension operator
satisfying \eqref{eq:S}. Then, the bounded linear operators
\begin{equation}\label{eq:subtractH}
	\Ascr = \Ascr_0-S\circ \Hscr:  H^\infty_M(C) \to H^\infty_M(A), \quad  
	\Bscr = \Bscr_0-S\circ \Hscr:  H^\infty_M(C) \to H^\infty_M(B)
\end{equation}
satisfy $c=\Bscr(c) - \Ascr(c)$ for every $c\in H^\infty_M(C)$.

The same result holds for a trivial bundle $X\times \C^N$ by applying it componentwise.
Any holomorphic vector bundle $E\to X$ over a Stein manifold embeds as a holomorphic 
vector subbundle of the trivial bundle $X\times\C^N$ of sufficiently large rank $N$,
and we have a holomorphic direct sum $X\times \C^N=E\oplus E'$ for some complementary
holomorphic vector subbundle $E'$ of $X\times \C^N$. This gives a holomorphic vector bundle 
projection $\theta: X\times \C^N \to E$ with the kernel $E'$.
It now suffices to apply the lemma for $X\times \C^N$ and postcompose the resulting 
operators by the projection $\theta$. 
\end{proof}

%
%
\begin{remark}\label{rem:shrinking}
If $D'\Subset D=A\cup B$ is a slightly smaller domain and we set 
$A'=A\cap D'$, $B'=B\cap D'$, then we can drop the condition (iii) on $M$
and obtain bounded linear operators
\[
	\Ascr : \Gamma^\infty_M(C,E) \to \Gamma^\infty_M(A',E),\qquad
	\Bscr : \Gamma^\infty_M(C,E) \to \Gamma^\infty_M(B',E)
\]
such that $c=\Bscr(c) - \Ascr(c)$ holds on $C'=A'\cap B'$ for each $c\in \Gamma^\infty_M(C,E)$.
To see this, we find operators $\Ascr_0$ and $\Bscr_0$ as before (see \eqref{eq:A0B0}), but in the last step
\eqref{eq:subtractH} we use a bounded linear extension operator 
$S: H^\infty(D\cap M) \to H^\infty(D')$ furnished by \cite[Lemma 3.1]{ForstnericPrezelj2001}.
The construction of such an operator uses Cartan's extension theorem and the Bergman space 
$L^2(D')\cap \Oscr(D')$, and it applies to any closed complex subvariety $M$ of $X$. 
\qed\end{remark}

%
%
\begin{remark}\label{rem:FP3}
Note that \cite[Lemma 3.2]{ForstnericPrezelj2001} coincides with the special case of 
Lemma \ref{lem:linearsplitting} for functions vanishing on $M$, but with condition (iii) 
on $M$ replaced by the following one:
\begin{enumerate}
\item[\rm (iii')] $\overline {M\cap C} \subset D=A\cup B$. Equivalently, $M$ does not intersect $bD\cap bC$. 
\end{enumerate}
I have recently realized that this extra assumption does not allow the geometric construction in the first part 
of the proof of \cite[Theorem 1.4]{ForstnericPrezelj2001} (see p.\ 62). (The second part of the said proof,
in which the section $f_0$ is assumed to be holomorphic on a neighbourhood of $K\cup X_0$ in $X$,
where $X_0$ is a subvariety of $X$, is fine.) This can be amended in two ways. We can allow  
small shrinking of domains as described in Remark \ref{rem:shrinking}, 
which makes condition (iii') superfluous; this suffices for the construction in \cite{ForstnericPrezelj2001}.
Alternatively, a section $f_0$ which is holomorphic on a subvariety $X_0$
and on a neighbourhood of a compact $\Oscr(X)$-convex set $K\subset X$ extends (by approximation on $K$)
to a section which is holomorphic on a neighbourhood of $K\cup X_0$ in $X$, and this does not require any
conditions on the submersion $h:Z\to X$ in question; see \cite[Theorem 3.4.1]{Forstneric2017E}. 
The latter approach has been used in our subsequent constructions in Oka theory, and in particular
in the monograph \cite{Forstneric2017E}.
\qed\end{remark}

The following nonlinear splitting lemma will be used in the proof of Theorem \ref{th:tubular}.

%
%
%
%
%
\begin{theorem}
\label{Acta:theorem4.1}
Let $X$ be a complex manifold with a Riemannian distance function $\dist$, 
$M$ be a closed complex submanifold of $X$, and $(A,B)$ be a Cartan pair in $X$. 
Given an open set $\wt C \subset X$ containing $C$, there are open sets $A'\supset A$, $B'\supset B$
with $C' = A'\cap B'\subset \wt C$ and constants $\epsilon_0>0, c_0>0$ such that the following holds. 
For every injective holomorphic map $\gamma : \wt C\to X$ with 
$\epsilon=\dist_{\wt C}(\gamma,\Id) < \epsilon_0$ and $\gamma|_{\wt C\cap M}=\Id$
there exist injective holomorphic maps $\alpha : A'\to X$, $\beta: B'\to X$ which depend continuously 
on $\gamma$, they agree with the identity map on $A'\cap M$ and $B'\cap M$, respectively, and satisfy
\begin{equation}\label{eq:linearestimate}
    \gamma\circ\alpha = \beta \ {\rm on\ } C',\quad\ 
    \dist_{A'}(\alpha,\Id) < c_0\epsilon, \quad\  \dist_{B'}(\beta,\Id) < c_0\epsilon.
\end{equation}
\end{theorem}

%
%
\begin{remark}\label{rem:variety}
This is a version of \cite[Theorem 9.7.1]{Forstneric2017E} with interpolation of the identity map on
a closed complex submanifold. The cited theorem is an improvement of 
\cite[Theorem 4.1]{Forstneric2003AM}, with linear estimates in \eqref{eq:linearestimate} and 
a substantially simpler proof than the one in \cite{Forstneric2003AM}. 

The original proof of the splitting theorem \cite[Theorem 4.1]{Forstneric2003AM} can be used to 
obtain Theorem \ref{Acta:theorem4.1} in the more general case when $M$ is a complex 
subvariety with singularities in $X$; see Remark \ref{rem:shrinking}. Since this generalization will not be 
needed in the present paper, we leave it to an interested reader.
\qed\end{remark}

\begin{proof}
We shall adapt the proof of  \cite[Theorem 9.7.1]{Forstneric2017E} to our current situation,
using this opportunity to explain in a better way a certain point in the proof.

By \cite[Proposition 5.7.3]{Forstneric2017E} there are smoothly bounded open sets 
$A'\supset A$, $B'\supset B$ such that $(\overline{A'},\overline {B'})$ is a Cartan pair,
the closure of $C'=A'\cap B'$ is contained in the open set $\wt C$, and 
$D'=A'\cup B'$ is a smoothly bounded strongly pseudoconvex domain.
We may furthermore choose $D'$ such that $M$ intersects $bD'$ transversely.
We replace $X$ by a Stein neighbourhood of $\overline {D'}$ and shrink $\wt C$ around $\overline{C'}$
if necessary. By a standard method, using compositions of flows of holomorphic vector fields on $X$,
we find a large integer $N\in\N$, an open neighbourhood $\Omega\subset X\times \C^N$
of the zero section $X\times\{0\}^N$, and a holomorphic map $\Phi : \Omega\to X$
such that for every $x\in X$ we have $\Phi(x,0)=x$ and
\[
	\text{the map}\ \ \Theta(x) := \frac{\di}{\di z}\Phi(x,z)\big|_{z=0} :\C^N\to T_xX\ \ \text{is surjective}.
\]
(A map $\Phi$ with these properties is called a {\em local dominating holomorphic spray} on $X$.)
We can split $X\times \C^N=E\oplus E'$, 
where $E'$ is the kernel of $\Theta$ and $E$ is some complementary 
holomorphic vector subbundle, which is isomorphic to the tangent bundle $TX$.
After shrinking $\Omega$ around the zero section, $\Phi_x:E_x\cap \Omega_x \to X$
is a biholomorphic map onto an open neighbourhood of $x$ in $X$ for every $x\in X$. Since this
holds for every subbundle $E\subset X\times \C^N$ complementary to $E'$, one easily sees
the following. 

%
%
\begin{lemma}\label{lem:transition}
Given a relatively compact domain $U\Subset X$ there exist constants $\epsilon_0>0$ 
and $d_0>1$ such that the following holds. For any pair of holomorphic maps $f,g:U\to X$ such that
$\dist(f,\Id):=\sup_{x\in U}\dist(f(x),x)<\epsilon_0$ and $\dist(g,\Id)<\epsilon_0$
there exists a unique holomorphic section $\xi\in \Gamma^\infty(U,E)$ such that for every $x\in U$
we have that 
\[
	\Phi(f(x),\xi(x)) = g(x)\ \ \ \text{and}\ \ \ d_0^{-1}|\xi(x)| \le \dist(f(x),g(x)) \le d_0 |\xi(x)|.
\]
\end{lemma}

Taking $f=\Id_U$, the above lemma identifies a uniform neighbourhood of the identity
map on $U$ in the space $\Oscr(U,X)$ with a neighbourhood of the zero section
in the Banach space $\Gamma^\infty(U,E)$ (where we are using the notation \eqref{eq:notation}) by
\[
	\xi \in \Gamma^\infty(D,E) \mapsto f \in \Oscr(D,X),
	\qquad
	f_\xi(x)=\Phi(x,\xi(x)),\ \ \ x\in U.
\]
%
%
Similarly, a uniform neighbourhood of the identity map $\Id_U$ in the subspace 
\[ 
	\Oscr_M(U,X)=\{f\in \Oscr(U,X) : f(x)=x\ \text{for}\ x\in M\cap U\} 
\] 
is identified with a neighbourhood of the zero section in the Banach space $\Gamma^\infty_M(U,E)$.
Let
\[
	\Ascr : \Gamma^\infty_M(C',E) \to \Gamma^\infty_M(A',E),\qquad
	\Bscr : \Gamma^\infty_M(C',E) \to \Gamma^\infty_M(B',E)
\]
be bounded linear operators furnished by Lemma \ref{lem:linearsplitting}. 
If $\xi \in \Gamma^\infty_M(C',E)$ is sufficiently close to the zero section, we define holomorphic 
maps $\alpha_\xi \in \Oscr_M(A',X)$ and $\beta_\xi\in \Oscr_M(B',X)$, 
close to the identity on their respective domains, by  
\begin{equation}\label{eq:alphaxi}
	\alpha_\xi(x) =\Phi(x,(\Ascr \xi)(x))\ \ (x\in A'),\quad \ \beta_\xi(x)  = \Phi(x,(\Bscr \xi)(x))\ \ (x\in B').
\end{equation}
In particular, we may assume that
$\alpha_\xi(C')\subset \wt C$, and hence $\gamma\circ\alpha_\xi \in \Oscr_M(C',X)$ is
close to the identity on $C'$. Note that $\alpha_0=\Id_{A'}$ and $\beta_0=\Id_{B'}$.
Assuming that $\gamma\in \Oscr_M(\wt C,X)$ is sufficiently close to the identity on $\wt C$ and 
$\xi \in \Gamma^\infty_M(C',E)$ is small enough, Lemma \ref{lem:transition} gives 
a unique section $\tilde \xi \in \Gamma^\infty_M(C',E)$ such that 
\begin{equation}\label{eq:Phi-main}
	\Phi(\gamma\circ\alpha_\xi(x),\tilde \xi(x)) = \beta_\xi(x),\quad\ x\in C'.
\end{equation}
This defines a smooth map $\phi(\gamma,\xi)=\tilde \xi$ from a neighbourhood of 
$(\Id,0)$ in $\Oscr_M(\wt C,X)\times \Gamma^\infty_M(C',E)$ into  $\Gamma^\infty_M(C',E)$
such that $\phi(\Id,0)=0$ and $\di_\xi \phi(\Id,\xi)|_{\xi=0} = \Id$; see \cite[Eq.\ (9.20)]{Forstneric2017E}.
By the implicit function theorem, the equation $\phi(\gamma,\xi)=0$ admits a unique local solution 
$\xi=\psi(\gamma)\in \Gamma^\infty_M(C',E)$ for $\gamma\in \Oscr_M(\wt C,X)$ near $\Id$, 
with $\psi(\Id)=0$. From \eqref{eq:Phi-main} we see that the maps
$
	\alpha_{\psi(\gamma)}\in \Oscr_M(A',X)
$
and $\beta_{\psi(\gamma)}\in \Oscr_M(B',X)$  \eqref{eq:alphaxi} satisfy 
\[
	\gamma\circ \alpha_{\psi(\gamma)}  =
	\Phi(\gamma\circ \alpha_{\psi(\gamma)},0) = \beta_{\psi(\gamma)}
	\ \ {\rm on}\ \ C'.
\]
The estimates in \eqref{eq:linearestimate} follow by observing 
that $\gamma\mapsto \alpha_{\psi(\gamma)}$ and $\gamma\mapsto \beta_{\psi(\gamma)}$ 
are given by smooth (not necessarily linear) operators. 
The maps $\alpha_{\psi(\gamma)}$ and $\beta_{\psi(\gamma)}$ 
need not be biholomorphic; however, the estimates \eqref{eq:linearestimate} imply that they are 
biholomorphic on a pair of slightly smaller domains provided $\gamma$ is close enough to the
identity on $\wt C$.  
\end{proof}

%
%
\section{Proof of Theorem \ref{th:tubular}}\label{sec:proof-tubular}

By Theorem \ref{th:Steinneighbourhoods} and Lemma \ref{lem:admissible} we may assume that the manifold
$X$ is Stein, the set $K$ is $\Oscr(X)$-convex, and the Stein submanifold $M$ is closed in $X$. 
Let $\Omega_0 \subset X$ and $\Phi_0:\Omega_0 \cup M\to Y$ be as in Theorem \ref{th:tubular}. 
We shall focus on the case when $\Phi_0$ is injective and $f=\Phi_0|_M:M\to f(M)=N\subset Y$
is an embedding; an obvious minor modification of the proof will apply to immersions.

A Stein neighbourhood $\Omega\subset X$ of $K\cup M$ and a 
biholomorphic map $\Phi:\Omega\to \Phi(\Omega)\subset Y$ satisfying the theorem 
will be constructed by a sequence of extensions of the initial map $\Phi_0$ to larger domains in $X$, 
which are open neighbourhoods of sets of the form $K\cup L$, where $L\subset M$ is a compact,
strongly pseudoconvex, and $\Oscr(M)$-convex domain containing $K\cap M$ in its relative interior.
At every step we match the given map $\Phi_0$ on $M$ and
approximate the previous map in the sequence on the previous domain. 
Every map in the sequence is covered along $M$ by an isomorphism of the 
normal bundles $\nu_{M,X}\to \nu_{N,Y}$ which is homotopic to the initial one, $\Theta$.
The scheme of proof follows that of \cite[Proposition 5.12.1]{Forstneric2017E}; see in particular 
\cite[p.\ 251]{Forstneric2017E}. The induction consists of two kinds of steps: 
the {\em noncritical case} and the {\em critical case}. We begin by describing the former.

Pick a strongly plurisubharmonic Morse exhaustion function $\rho:M\to\R$ such that 
$\rho<0$ on $K\cap M$, $\rho>0$ on $M\setminus \overline\Omega_0$, and $0$ is a regular
value of $\rho$. Choose a compact $\Oscr(X)$-convex set $K'\subset \Omega_0$
containing $K$ in its interior and such that $\rho<0$ on $K'$.
The noncritical case amounts to the following lemma, which will be
used with different pairs of sets $K$ and $L$.

%
%
\begin{lemma}\label{lem:noncritical}
(Assumptions as above). Assume that $\rho$ has no critical values on $[0,c]$ for some $c>0$.
Set $L=\{x\in M:\rho(x)\le c\}$. Then we can approximate $\Phi_0$ as closely as desired uniformly
on $K'$ by a biholomorphic map $\Phi_1:\Omega_1\to \Phi_1(\Omega_1)\subset Y$
from a neighbourhood $\Omega_1\subset X$ of $K'\cup L$ such that $\Phi_1$ agrees with 
$\Phi_0$ on $M\cap \Omega_1$.
\end{lemma}

\begin{proof}
By \cite[Lemma 5.10.3]{Forstneric2017E} there is a finite increasing sequence of 
compact, strongly pseudoconvex and $\Oscr(M)$-convex domains in $M$:
\[
	\{\rho\le 0\} = L_0 \subset L_1 \subset \cdots \subset L_k=L=\{\rho\le c\},
\]
such that for every $i=0,\ldots,k-1$ we have $L_{i+1}=L_i\cup B_i$, where $B_i\subset M$ is a 
{\em convex bump} attached to $L_i$. This means that $(L_i,B_i)$ is a Cartan pair
(see Def.\ \ref{def:Cartan-pair}), and there is a holomorphic chart on a neighbourhood of 
$B_i$ in $M$ in which $B_i$ and $C_i=L_i\cap B_i$ are strongly convex domains in $\C^m$
with $m=\dim M$. 
Set $A_i=K'\cup L_i$. By Lemma \ref{lem:hc}, $(A_i,B_i)$ is a Cartan pair
and $C_i=A_i\cap B_i =L_i\cap B_i \subset M$.
Note that $K'\cup L_0=A_0\subset A_1\subset \cdots \subset A_k=K'\cup L$.

We prove the lemma by a finite induction. Assume that for some $i\in \{0,1,\ldots,k-1\}$
we have a neighbourhood $U_i\subset X$ of $A_i$ and a biholomorphic map 
$F_i:U_i\to V_i=F_i(U_i)\subset Y$ which agrees with $\Phi_0$ on $M\cap U_i$.
We wish to find the next biholomorphic map $F_{i+1}:U_{i+1}\to V_{i+1}\subset Y$ 
on a neighbourhood $U_{i+1}\subset X$ of $A_{i+1}$ which agrees with $\Phi_0$ on $M\cap U_{i+1}$
and approximates $F_i$ on $K'$ as closely as desired. Assuming that the induction works,
the map $F_k$ can then be taken as $\Phi_1$ in the lemma.

Fix $i\in \{0,1,\ldots,k-1\}$ and write $A_i=A$, $B_i=B$, $C_i=C$, and $F_i=F$. 
Let $m=\dim M$. The conditions imply that there are holomorphic coordinates 
\begin{equation}\label{eq:localcoord}
	z=(z',z''),\quad z'=(z_1,\ldots, z_m)= x'+\imath y',\quad z''=(z_{m+1},\ldots,z_n)=x''+\imath y''
\end{equation}
on an open neighbourhood $W$ of $B$ in $X$ 
such that $M\cap W=\{z''=0\}$ and the sets $B\subset C$ are strongly convex
domains in $\C^m_{z'}\times \{0\}^{n-m}$. Set $d=n-m$. 
By \cite[Proposition 5.7.3]{Forstneric2017E} we can approximate $(A,B)$ from the outside 
by a Cartan pair $(A',B')$ consisting of compact strongly pseudoconvex domains in $X$ such that 
$F$ is holomorphic on a neighbourhood of $A'$ and $B'\subset W$. 
Likewise, there are local holomorphic coordinates $w=(w',w'')$ around the set 
$F(B) = \Phi(B) \subset N=\Phi(M)$ in $Y$ such that $N$ corresponds to $w''=0$
and $F|_B$ is the identity map $(z',0'') \mapsto (z',0'')$. In this pair of coordinates, $F$ is of the form
\begin{equation}\label{eq:F}
	F(z',z'')  = \bigg( z' + \sum_{j=m+1}^n z_j \, G'_j(z',z''), \sum_{j=m+1}^n z_j \, G''_j(z',z'') \bigg)
\end{equation}
where the holomorphic map $G''=(G''_{m+1},\ldots, G''_n)$ with values in $\C^{d^2}$ has maximal rank $d=n-m$ 
at every point $(z',0'') \in C\times \{0''\}$, and hence in a neighbourhood of this set. By shrinking $A'$ and 
$B'$ around $A$ and $B$, respectively, we may assume that this holds on $C'=A'\cap B'$. 
In other words, $G''$ assume values in $GL_d(\C)$. By a theorem due to 
Grauert \cite{Grauert1957I} (see also \cite[Proposition 5.6.1]{Forstneric2017E}), 
we can approximate $G''$ as closely as desired on a neighbourhood of $C'$ by a holomorphic map 
$\wt G'':B'\to GL_d(\C)$. Similarly, we can approximate the holomorphic map $G'=(G'_{m+1},\ldots,G'_n)$ 
with values in $\C^{dm}$ on a neighbourhood of $C'$ by a holomorphic map $\wt G'$ on a 
neighbourhood of $B'$. Inserting these two maps into the equation \eqref{eq:F} defines a 
holomorphic map $\wt F$ on a neighbourhood of $B'$
which is close to $F$ on a neighbourhood of $C'$, it agrees with $F$ on $M\cap C'$, and 
is biholomorphic on a neighbourhood of $B$ in $X$. On a neighbourhood of $C'$ we thus have  
\[
	\wt F = F\circ \gamma,
\]
where $\gamma$ is an injective holomorphic map close to the identity on a neighbourhood of $C'$
which agrees with the identity on $C'\cap M$. Assuming that the approximations are close enough,
Theorem \ref{Acta:theorem4.1} gives a splitting 
\[
	\text{$\gamma\circ\alpha = \beta$\ \ \ on $C'$}, 
\]
where $\alpha:A'\to X$ and $\beta:B'\to X$ are injective holomorphic maps 
close to the identity, which agree with the identity on $A'\cap M$ and $B'\cap M$, respectively. Then, 
\[
	\wt F\circ \alpha = F\circ \beta\quad \text{on $C'$}.
\]
This defines the next holomorphic map $F_{i+1}$ in the sequence which is injective holomorphic on a neighbourhood
of $A_{i+1}=K'\cup L_{i+1}$, it approximates $F=F_i$ as closely as desired on $K'$, 
and it agrees with $\Phi_0$ on the intersection of its domain with $M$. This completes
the induction step and hence proves the lemma. 
\end{proof}

%
%
We now consider the {\em critical case}, which amounts to passing a critical value $c_0>0$ 
of the Morse exhaustion function $\rho:M\to \R$. Assume inductively that $F:W_0\to F(W_0)\subset Y$ 
is an injective holomorphic map on a neighbourhood $W_0\subset X$ 
of $K'\cup L$, where $L=\{\rho\le c\}$ for some $c<c_0$ close to $c_0$ such that $\rho$ has no 
critical values in $[c,c_0)$ and there is precisely one (Morse) critical point $p_0\in M$ with $\rho(p_0)=c_0$.
%
%
Let $k\in \{0,1,\ldots,m\}$ denote the Morse index of $\rho$ at $p_0$.
By \cite[Lemma 3.10.1]{Forstneric2017E}, there are local holomorphic coordinates 
\eqref{eq:localcoord} on a convex neighbourhood $U\subset X$ of the point $p_0$
(i.e., such that $z(U)\subset \C^n$ is convex), with $z(p_0)=0$ and $M\cap U=\{z''=0\}$, such that
\[
	\rho(z') = c_0 - \sum_{j=1}^k x_j^2 + \sum_{j=k+1}^m x_j^2 + 
	\sum_{j=1}^m \lambda_j y_j^2 + o(|z'|^2),
\]
where $\lambda_j\ge 1$ for all $j$ and $\lambda_j > 1$ for $j=1,\ldots,k$.
Furthermore, we can choose $\rho$ such that the remainder term $o(|z'|^2)$
vanishes (see \cite[Lemma 3.10.3]{Forstneric2017E}). 
The normal bundle to $M$ in these coordinates is the $z''$-space.

If $k=0$ then $p_0$ is a local minimum of $\rho$, and in this case a new connected component
of the sublevel set $\{\rho<c\}$ appears at $p_0$ as $c$ passes the value
$\rho(p_0)=c_0$. We can extend $F$ as a biholomorphic map on the product 
of this new component with a ball in the $z''$-direction such that $F$ agrees with 
$\Phi_0$ on the $z'$-space and its normal component is determined by the isomorphism 
$\Theta$ of the normal bundles. 

Assume now that $k\in \{1,\ldots,m\}$. Let the constant $c<c_0$ be as above. Set
\[
	E=\Big\{(x'+\imath 0',0''): \sum_{j=1}^k x_j^2\le c_0-c,\ x_{k+1}=\cdots=x_m=0\Big\}.
\]
This is a linear totally real $k$-disc attached with its boundary sphere $bE\cong S^{k-1}$
to the boundary $bL=\{\rho = c\}$ of the domain $L=\{\rho\le c\}\subset M$. Choosing local coordinates 
$w$ on the target side as above, the map $F$ is of the form \eqref{eq:F} on a 
neighbourhood of $K'\cup L$. Let $W$ be a compact strongly pseudoconvex domain such that 
$K'\cup L\subset W\subset W_0$, $bL=\{\rho=c\} \subset bW$, and $E$ intersects $bW$ 
transversely along the sphere $bE\subset bW$. 
In fact, we can extend $\rho$ to a strongly plurisubharmonic function on $U$ by setting 
\[
	\wt \rho(z)=\rho(z') + C|z''|^2
\]
for a big $C>0$ and choose $W$ such that $W\cap U=\{\wt\rho \le c\}$. 
The situation is illustrated in \cite[Fig.\ 5.3, p.\ 246]{Forstneric2017E}.
Pick a small $s>0$ and let $\Sigma$ denote the linear totally real strip 
\[
	\Sigma=\Big\{(x'+\imath 0',x''+\imath 0'') \in U: x_{k+1}=\cdots=x_m=0,\ \ 
	\sum_{j=m+1}^n x_j^2 \le s \Big\}.
\]
If $s>0$ is small enough then $S=W\cup \Sigma$ is a Stein compact in $X$.
(See \cite[Lemma 3.9.3]{Forstneric2017E} which is due to Eliashberg \cite{Eliashberg1990}. 
More precise constructions of this kind can be found in the monograph \cite{CieliebakEliashberg2012}
by Cieliebak and Eliashberg and in the paper \cite{ForstnericKozak2003}.)  
This means that $S$ is a {\em strongly admissible set} 
in the sense of \cite[Definition 5, p.\ 156]{FornaessForstnericWold2020}. 

By using the isomorphism $\Theta$ of the complex normal bundles over $F$
(which in the given local coordinates correspond to the $z''$-space), 
we can extend $F$ to a smooth map on $S=W\cup \Sigma$
which agrees with the given map on a neighbourhood of $W$ and on 
$M$, and it is a smooth injective immersion on the strip $\Sigma\setminus W$.
By the version of Mergelyan theorem in \cite[Corollary 9, p.\ 178]{FornaessForstnericWold2020}, 
we can approximate this extended map $F$ in the $\Cscr^r$-topology on $S\cap U$ for any $r\in\N$
by a holomorphic map $\wt F$ on a neighbourhood of $W\cap U$ such that $\wt F=\Phi_0$ holds
on $M\cap U$. (In order to ensure this interpolation condition, we actually approximate the map 
$G=(G',G'')$ which appears in \eqref{eq:F}, like we did in the noncritical
case but now using the cited Mergelyan theorem. As in the noncritical case, 
$G''$ assumes values in $GL_d(\C)$.) Assuming that the approximation is close
enough in $\Cscr^1(S\cap U)$, the map $\wt F$ is injective holomorphic
on a neighbourhood of $S\cap U$. 
We then form a Cartan pair $(A,B)$ where $A=W$, $B$ is a neighbourhood of 
$\overline {\Sigma\setminus W}$, and $D=A\cup B$ is a strongly pseudoconvex 
neighbourhood of $S=W\cup \Sigma$. (Such Cartan pairs were constructed by
Henkin and Leiterer in \cite{HenkinLeiterer1998}.) Finally, we use Theorem \ref{Acta:theorem4.1} 
to glue $F$ with $\wt F$ into an injective holomorphic map $F'$ on a neighbourhood of $S$ which 
agrees with $\Phi_0$ on $M$. The details of this gluing are essentially the same as in the 
noncritical case explained above.

This procedure extends the map to a neighbourhood of domain in $M$ obtained by attaching to 
$\{\rho\le c\}$ a handle which effects 
the change of topology of the sublevel set of $\rho$ at the critical point $p_0$. 
In order to extend the map to a neighbourhood of a sublevel set $\{\rho\le c'\}$ for some 
$c'>\rho(p_0)$, we apply the noncritical case with a different strongly plurisubharmonic function 
$\tau$ as described in \cite[p.\ 247]{Forstneric2017E}, where the appropriate function $\tau$ is given
 by \cite[Lemma 3.11.1]{Forstneric2017E}. 

To prove Theorem \ref{th:tubular}, we alternately use the noncritical and the critical case
developed above in a well known procedure; see e.g.\ \cite[proof of Proposition 5.12.1]{Forstneric2017E}.

%
%
%
%
\section{Proof of Theorem \ref{th:main} and a generalization} 
\label{sec:proof}

In this section we prove Theorem \ref{th:main}. We then show that essentially the same proof
applies to any Stein manifold with the density property in place of the Euclidean space $\C^n$. 

We recall the following special case of \cite[Theorem 15]{ForstnericRitter2014}.
In order to adjust the notation to our situation, the Stein manifold $X$ in the cited paper is now called $M$,
the compact set $K\subset X$ is replaced by $K'$, and the subvariety $X'$ is empty.

%
%
\begin{theorem}\label{th:FR2014}
Let $L$ be a compact polynomially convex set in $\C^n$ for some $n>1$. Let $M$ be a Stein manifold
with $2\dim M\le n$, $K'$ be a compact $\Oscr(M)$-convex set, $U\subset M$ be an open set containing $K'$, 
and $f :U \to\C^n$ be a holomorphic map such that $f(bK') \cap L =\varnothing$. 
Then for every $\epsilon >0$ there exists a proper holomorphic immersion $F: M\to \C^n$ 
(embedding if $2\dim M+1\le n$) satisfying 
\[
	(a)\ F(M\setminus K') \subset \C^n\setminus L, \qquad (b)\ ||F-f||_{K'}<\epsilon.
\]
\end{theorem}

As pointed out in \cite[Remark 1.3]{Forstneric2019JAM}, the proof of this result in 
\cite{ForstnericRitter2014} is incomplete in the borderline case $2\dim M =n$, 
but it is completed by applying \cite[Lemma 3.1]{Forstneric2019JAM}.  
 
\begin{proof}[Proof of Theorem \ref{th:main}] 
By Lemma \ref{lem:admissible}, we may assume that the manifold $X$ is Stein, $M$ is a closed 
Stein submanifold of $X$, and $K$ is $\Oscr(X)$-convex.
Let $\Omega_0\subset X$ be a neighbourhood of $K$ and 
$\Phi_0:\Omega_0\to \Phi_0(\Omega_0)\subset\C^n$ be a biholomorphic map as in the theorem
such that $L:=\Phi_0(K)\subset \C^n$ is polynomially convex. Choose a slightly bigger compact 
$\Oscr(X)$-convex domain $K_1\subset \Omega_0$ containing $K$ 
in its interior and set $K'=K_1\cap M$ and $U=\Omega_0\cap M$. The holomorphic embedding 
$f:=\Phi_0|_U:U\to\C^n$ then satisfies the assumptions of Theorem \ref{th:FR2014}.
An inspection of the proof of this result in \cite{ForstnericRitter2014} shows that 
there is a proper holomorphic immersion $F:M\to \C^n$ (embedding if $2\dim M+1\le n$)
satisfying conditions (a) and (b) in the theorem such that $F$ is also  
holomorphic and uniformly close to $\Phi_0$ on a neighbourhood of the set $K_1$ in $X$.
This is seen by the standard construction in Oka theory; cf.\ \cite[proof of Proposition 5.12.1]{Forstneric2017E}. 
The dimension condition $2\dim M\le n$ and the hypothesis that the restricted tangent
bundle $TX|_M$ is trivial implies by Corollary \ref{cor:stablenormabundle} that 
the immersion $F|_M:M\to\C^n$ is covered by an isomorphism of the normal bundles
$\nu_{M,X}\to \nu_F$, which can be chosen to agree with the one induced by $dF$ 
over $K_1\cap M$. Hence, the conclusion of the theorem follows by applying
Theorem \ref{th:tubular} to extend $F$, by approximation on $K_1$, 
to a holomorphic immersion $\Phi:\Omega\to\C^n$
(embedding if $2\dim M+1\le n$) from a neighbourhood of $K_1\cup M$ in $X$. 
\end{proof}

%
%
The above proof uses two main ingredients, namely Theorem \ref{th:FR2014} from 
the paper \cite{ForstnericRitter2014} (and its extension explained in the proof)
and Theorem \ref{th:tubular} proved in Section \ref{sec:proof-tubular}.
The former result has been generalized in the papers 
\cite{AndristForstnericRitterWold2016,Forstneric2019JAM} to the case when 
$\C^n$ is replaced by an arbitrary $n$-dimensional Stein manifold with the 
holomorphic density or volume density property. Let us recall the definition. 

A holomorphic vector field on a complex manifold is said to be complete if its 
flow is defined for all complex values of time. A complex manifold $Y$ is said to enjoy 
the {\em density property} if every holomorphic vector field on $Y$ can be approximated 
uniformly on compacts by sums of $\C$-complete holomorphic vector fields. 
The {\em volume density property} refers to the analogous property for holomorphic
vector fields having divergence zero with respect to a holomorphic volume form
on the manifold. This notion, introduced by Varolin in \cite{Varolin2000,Varolin2001}, plays a major role
in contemporary complex analysis. On Stein manifolds, the density property implies the 
{\em Andersen-Lempert property} on approximation of isotopies of biholomorphic maps 
on compact holomorphically convex sets by holomorphic automorphisms of the manifold
(see \cite[Theorem 4.10.5]{Forstneric2017E}). In many respects, Stein manifolds with the 
density (or the volume density) 
property behave like Euclidean spaces, except that they may be topologically nontrivial. 
For surveys of this topic, see \cite[Sect.\ 10]{Forstneric2017E} 
and the paper by Kaliman and Kutzschebauch \cite{KalimanKutzschebauch2015}. 
A list of examples of Stein manifolds with the density property is also provided 
in \cite[Example 1.3]{AndristForstnericRitterWold2016}. 

By following the proof of Theorem \ref{th:main} but replacing the use of Theorem \ref{th:FR2014}
by the analogous results in \cite{AndristForstnericRitterWold2016,Forstneric2019JAM},
one obtains the following generalization of Theorem \ref{th:main}.
We leave the details of the proof to the reader. 

%
%
\begin{theorem}\label{th:density}
Let $X$ be a complex manifold, $M$ be a Stein submanifold of $X$, and $K$ be a compact set
such that $K\cap M$ is a compact $\Oscr(M)$-convex set. Assume that 
\begin{enumerate}[\rm (a)]
\item $Y$ is a Stein manifold of dimension $\dim Y=\dim X>1$ having the density or the volume density property, 
\item $\Omega_0\subset X$ is an open neighbourhood of $K$
and $\Phi_0:\Omega_0\stackrel{\cong}{\to} \Phi_0(\Omega_0)\subset Y$ is a biholomorphic 
map such that $\Phi_0(K)$ is $\Oscr(Y)$-convex, and 
\item $\Phi_0|_{M\cap\, \Omega_0}$ extends to a continuous map $f:M\to Y$,  
covered by a continuous complex vector bundle isomorphism $\Theta:TX|_M\to f^*(TY)$
which agrees over $M\cap \Omega_0$ with the tangent map 
$T\Phi_0:TX|_{M\cap \,\Omega_0} \to (\Phi_0|_{M\cap \, \Omega_0})^* TY$.
\end{enumerate}
If $\dim X\ge 2\dim M+1$ then for any $\epsilon>0$ there are a Stein domain $\Omega \subset X$ 
containing $K\cup M$ and a biholomorphic map $\Phi:\Omega \stackrel{\cong}{\to} \Phi(\Omega) \subset Y$ 
such that 
\begin{enumerate}[\rm (i)]
\item $\sup_{x\in K}\dist(\Phi(x),\Phi_0(x)) <\epsilon$, where $\dist$ is a Riemannian distance function on $Y$,  
\item $\Phi(M)$ is a closed complex submanifold of $Y$, and
\item $\Phi|_M$ is homotopic to $f$ by a homotopy which is uniformly close to $f$ on $K\cap M$.
\end{enumerate}
If $\dim X=2\dim M$, there are a domain $\Omega\subset X$ as above and a holomorphic immersion 
$\Phi:\Omega\to Y$ satisfying conditions (i) and (iii) such that $\Phi|_M:M\to Y$ is proper. 
\end{theorem}

Further improvements of Theorems \ref{th:main} and  \ref{th:density}
are possible in view of the recent results of Kusakabe 
\cite[Theorem 1.2 and Corollary 1.3]{Kusakabe2020complements}
to the effect that the complement of any compact polynomially convex set in $\C^n$
for $n>1$ is an Oka manifold, and the analogous result holds for holomorphically
convex sets in any Stein manifold with the density or volume density property of dimension $n>1$.
(See also \cite{ForstnericWold2020MRL} for another proof of these results.) Together with 
the approximation and gluing constructions in Oka theory, this allows for an unprecedented 
control of the range of holomorphic maps.
Let us illustrate this phenomenon in the context of Theorem \ref{th:main}. We have the following result.

%
%
\begin{theorem}\label{th:schlicht}
Assume that $X$ is a Stein manifold, $M$ is a closed complex submanifold of $X$
with $2\dim M+1\le n=\dim X$ such that $TX|_M$ is a trivial bundle, $K\subset X$ is a 
compact $\Oscr(X)$-convex set, $\Omega_0\subset X$ is an open neighbourhood of $K$,
and $\Phi_0:\Omega_0\to \Phi(\Omega_0)\subset \C^n$ is a biholomorphic map such that $\Phi_0(K)$ 
is polynomially convex. Then, there is a holomorphic map $\Phi:X\to\C^n$ which is injective holomorphic
on a neighbourhood of $M\cup K$ such hat $\Phi(M)$ is a closed complex submanifold of $\C^n$ and
\begin{equation}\label{eq:schlicht}
	\Phi(X\setminus K)  \subset \C^n\setminus \Phi(K).
\end{equation}
If $4\le 2\dim M=n=\dim X$ then there is a holomorphic map $\Phi:X\to\C^n$ which is an immersion
on a neighbourhood of $K\cup M$, it is proper on $M$, and it satisfies \eqref{eq:schlicht}.
\end{theorem}

The addition to Theorem \ref{th:main} is \eqref{eq:schlicht}, which says that
$\Phi:X\to \C^n$ covers the set $\Phi(K) \subset \C^n$ in a schlicht way. 
The analogous result holds in the context of Theorem \ref{th:density}.

Here is the idea of proof. By Lemma \ref{lem:admissible} we may enlarge $K$ slightly so that
it is the closure of a strongly pseudoconvex domain and $\Phi_0(K)$ is polynomially convex. 
By Theorem \ref{th:main} there are a neighbourhood $\Omega\subset X$ of $K\cup M$
and a biholomorphic map $\Phi:\Omega\to \Phi(\Omega) \subset \C^n$ such that 
$\Phi|_M:M\to\C^n$ is a proper holomorphic embedding, and $\Phi$ approximates $\Phi_0$
as closely as desired on a neighbourhood of $K$. Assuming that the approximation is close enough, the set 
$L=\Phi(K)\subset \C^n$ is also polynomially convex.
By using Kusakabe's theorem that $\C^n\setminus L$ is an Oka manifold,  we can interpolate $\Phi$ on $M$ 
and approximate it as closely as desired on a neighbourhood of $K$ by a holomorphic map $X\to\C^n$, 
still denoted $\Phi$, which also satisfies the schlichtness condition \eqref{eq:schlicht}. 
To do this, we perform a stepwise extension of $\Phi$ to larger domains by using methods of Oka theory 
(cf.\ the proof of \cite[Theorem 5.4.4]{Forstneric2017E}) so that the image of each added piece 
of the domain of the map lies in the Oka domain $\C^n\setminus L$. The idea is similar to the proof the 
Oka principle for removal of intersections, given by \cite[Theorem 7.2.1]{Forstneric2017E}.

An even stronger conclusion can be made if the map $\Phi$ in Theorem \ref{th:main} can be chosen
such that $A:=\Phi(M) \subset\C^n$ is an algebraic submanifold of $\C^n$.
Our assumptions imply that $\dim A\le n-2$. The following result is an immediate corollary to 
\cite[Theorem 1.6]{Kusakabe2020complements} and the general description of an 
algebraic subvariety of $\C^n$ by using proper projections onto linear subspaces 
of dimension $\dim A$ (see Chirka \cite[Theorem 2, p.\ 77]{Chirka1989}).

%
%
\begin{corollary}
If $A\subset \C^n$ is a closed algebraic subvariety with $\dim A\le n-2$ and $L$
is a compact polynomially convex set in $\C^n$, then $\C^n\setminus (A\cup K)$ is an Oka domain. 
Furthermore, $A\cup L$ has a basis of open Stein neighbourhoods $V \subset \C^n$ such that 
$\C^n\setminus \overline V$ is Oka.
\end{corollary}

By using this result, we can choose the extension $\Phi:X\to \C^n$ in Theorem \ref{th:schlicht} 
so that it covers a neighbourhood of $\Phi(M\cup K)=A\cup L$ in a schlicht way:
\[
	\Phi(X\setminus (M\cup K)) \subset \C^n\setminus (A\cup L). 
\]

%
%
\subsection*{Acknowledgements}
The author is supported in part by the research program P1-0291 and grant J1-9104
from ARRS, Republic of Slovenia. He wishes to thank the referee for a careful reading 
of the paper and for catching several misprints.






\vspace*{5mm}
\noindent Franc Forstneri\v c

\noindent Faculty of Mathematics and Physics, University of Ljubljana, Jadranska 19, SI--1000 Ljubljana, Slovenia, and 

\noindent 
Institute of Mathematics, Physics and Mechanics, Jadranska 19, SI--1000 Ljubljana, Slovenia

\noindent e-mail: {\tt franc.forstneric@fmf.uni-lj.si}

\end{document}